\documentclass[reqno]{amsart}

\usepackage{enumitem}
\usepackage{amssymb}
\usepackage{hyperref}

\newcommand*{\mailto}[1]{\href{mailto:#1}{\nolinkurl{#1}}}
\newcommand{\arxiv}[1]{\href{http://arxiv.org/abs/#1}{arXiv:#1}}

\newtheorem{theorem}{Theorem}[section]

\newtheorem{lemma}[theorem]{Lemma}

\newcommand{\R}{{\mathbb R}}

\newcommand{\Z}{{\mathbb Z}}

\newcommand{\cR}{{\mathcal R}}

\newcommand{\Kt}{\widetilde{K}}

\newcommand{\A}{\mathcal{A} }

\newcommand{\nn}{\nonumber}

\newcommand{\I}{\mathrm{i}}
\newcommand{\E}{\mathrm{e}}

\newcommand{\beq}{\begin{equation}}
\newcommand{\eeq}{\end{equation}}

\newcommand{\floor}[1]{\lfloor#1 \rfloor}

\newcommand{\vp}{\varphi}
\newcommand{\vpt}{\widetilde{\varphi}}

\numberwithin{equation}{section}

\begin{document}

\title[Scattering and Dispersion Estimates for Jacobi Operators]{Properties of the Scattering Matrix and Dispersion Estimates for Jacobi Operators}

\author[I.\ Egorova]{Iryna Egorova}
\address{B. Verkin Institute for Low Temperature Physics\\ 47, Lenin ave\\ 61103 Kharkiv\\ Ukraine}
\email{\href{mailto:iraegorova@gmail.com}{iraegorova@gmail.com}}

\author[M.\ Holzleitner]{Markus Holzleitner}
\address{Faculty of Mathematics\\ University of Vienna\\Oskar-Morgenstern-Platz 1\\ 1090 Wien\\ Austria}
\email{\href{mailto:amhang1@gmx.at}{amhang1@gmx.at}}

\author[G.\ Teschl]{Gerald Teschl}
\address{Faculty of Mathematics\\ University of Vienna\\
Oskar-Morgenstern-Platz 1\\ 1090 Wien\\ Austria\\ and International Erwin Schr\"odinger
Institute for Mathematical Physics\\ Boltzmanngasse 9\\ 1090 Wien\\ Austria}
\email{\mailto{Gerald.Teschl@univie.ac.at}}
\urladdr{\url{http://www.mat.univie.ac.at/~gerald/}}

\thanks{J. Math. Anal. Appl. {\bf 434}, 956--966 (2016)}
\thanks{{\it Research supported by the Austrian Science Fund (FWF) under Grants No.\ Y330, V120, and W1245}}


\begin{abstract}
We show that for a Jacobi operator with coefficients whose $(j+1)$'th moments are summable the $j$'th derivative of the scattering matrix is in the Wiener algebra of functions with summable Fourier coefficients. We use this result to improve the known dispersive estimates with integrable time decay for the time dependent Jacobi equation
in the resonant case.
\end{abstract}

\maketitle

%%%%%%%%%%%%%%%%%%%%%%%%%%%%%%%%%%%%%%%%%%%%%%%%%%%%%%%%%%%%%%%%%%%%%%%%%%
\section{Introduction}
%%%%%%%%%%%%%%%%%%%%%%%%%%%%%%%%%%%%%%%%%%%%%%%%%%%%%%%%%%%%%%%%%%%%%%%%%%%%%

This paper is concerned with Jacobi operators
\beq\label{jaceq}
H u(n)=a(n-1)u(n-1) + b(n) u(n) + a(n) u(n+1), \qquad n \in \Z,
\eeq
satisfying $a(n)\to 1/2$, $b(n)\to 0$ such that
\beq \label{mainas}
\Big|a(n)-\frac 1 2 \Big|+ |b(n)| \in \ell^1_{\sigma}(\Z)
\eeq
for some $\sigma\ge 1$. Here $\ell^p_\sigma(\Z)$ is the set of all doubly infinite sequences for which the norm
\[
   \Vert u\Vert_{\ell^p_{\sigma}}= \begin{cases} \left( \sum_{n\in\Z} (1+|n|)^{p\sigma} |u(n)|^p\right)^{1/p}, & \quad p\in[1,\infty),\\
   \sup_{n\in\Z} (1+|n|)^{\sigma} |u(n)|, & \quad p=\infty, \end{cases}
\]
is finite. The case $\sigma=0$ corresponds to the usual unweighted spaces $\ell^p_0\equiv\ell^p$.

The special case $a(n)\equiv 1$ is also known as the discrete Schr\"odinger operator.
If \eqref{mainas} holds for $\sigma=1$ it is well known (\cite{tjac}) that the spectrum of $H$ consists of a purely absolutely continuous part
covering $[-1,1]$ plus a finite number of eigenvalues located in $\R\setminus[-1,1]$.
In addition, there could be resonances at the edges of the continuous spectrum,
which means that there exist corresponding bounded solutions (and which is equivalent to the fact,
that the Wronskian of the two Jost solutions vanishes (cf.\ \cite[Definition 3.5, Lemma 3.6]{EKT})).
In case of such a resonance it is a delicate question to determine the behavior
of the scattering matrix near such a resonance under the minimal assumption $\sigma=1$.
From the analogous result for the continuous Schr\"odinger equation it is expected
that the scattering matrix is continuous at such a resonance and this was first established
in \cite{emt} (see also \cite{Kh5} for a special case). In fact, in \cite{EKT} this result was refined by showing that the scattering
coefficients are in the Wiener algebra of functions on the unit circle with integrable
Fourier coefficients. In the present paper we generalize this result by showing that if
\eqref{mainas} holds for $\sigma\ge 1$ then also all derivatives of order up to $\sigma-1$
of the scattering coefficients are again in the Wiener algebra. 

This question is not only of interest in scattering theory but also plays an important role in the solution of the Toda equation
via the inverse scattering transform (see e.g., \cite{KT1,KT2} where continuity of higher derivatives is needed)
and in proving dispersive estimates for the corresponding linear evolution equation
\beq
\frac{d}{dt} u(t) = H u(t).
\eeq
The latter case has attracted considerable interest recently (e.g.\ \cite{CT,EKT,PS} and the references therein) due to its importance for deriving asymptotic stability of
solitons for the associated nonlinear evolution equations (see e.g.\ \cite{KPS,K09,PS11,SK}).

As an application of our results we will establish two dispersive decay estimates. First of all we extend \cite[Theorem 5.3]{EKT} to the case of Jacobi operators.

\begin {theorem} \label{main1}
Let $H$ be defined by \eqref{jaceq} with $|a-\tfrac{1}{2}|+ |b| \in\ell^1_1(\Z)$. Then
\beq\label{Schr-as1}
\Vert \E^{-\I tH}P_c\Vert_{\ell^1\to \ell^\infty}=\mathcal{O}(t^{-1/3}),\quad t\to\infty,
\eeq
\beq\label{Schr-as2}
\Vert \E^{-\I tH} P_c\Vert_{\ell^2_\sigma\to \ell^2_{-\sigma}}=\mathcal{O}(t^{-1/2}),\quad t\to\infty,\quad\sigma>1/2.
\eeq
\end{theorem}

The following extends \cite[Theorem 6.1]{EKT} to the case of Jacobi operators and is new in the case of resonances.

\begin{theorem} \label{mainthmdisc}
Let $H$ be defined by \eqref{jaceq} with $|a-\tfrac{1}{2}|+ |b| \in\ell^1_j(\Z)$, where $j=3$ if there is a resonance at $\hat{z}=+1$ or $\hat z=- 1$ and $j=2$ if neither
of these points is a resonance. Furthermore, for every resonance $\hat{z}\in\{\pm1\}$ let $\vp_{\hat z}(n)$ be a bounded solution of $H \vp_{\hat z}(n)=\frac{\hat z + \hat z^{-1}}{2} \vp_{\hat z}(n)$ normalized according to $\lim_{n \to +\infty} (|\vp_{\hat z}(n)|^2 + |\vp_{\hat z}(-n)|^2)=2$. 
Denote by $P_{\hat z}$ the projection onto the span of $\vp_{\hat z}$ given by the kernel $[P_{\hat z}](n,k)=\vp_{\hat z}(n)\vp_{\hat z}(k)$ and set $P_{\hat z}=0$ if there is no resonance at $\hat z$. By $P_{ac}$ we denote the projector on the absolutely continuous subspace of $H$. Then the following decay holds:
\begin{align} \label{resdecdisc}
\Vert \E^{-\I t H}P_{ac}-\frac{\E^{-\I t}}{\sqrt{-2\pi \I t}}P_{1}-\frac{\E^{\I t}}{\sqrt{2\pi \I t}}P_{-1} \Vert_{\ell^1_2 \to \ell^\infty_{-2}}=\mathcal{O}(t^{-4/3}),\quad t\to\infty.
\end{align}
\end{theorem}

Such a decay estimate with integrable time decay has previously only been established in the non-resonant case where $P_{-1}=P_1=0$ (cf.\ \cite[Theorem 6.1]{EKT}).
For continuous one-dimensional Schr\"odinger equations an analogous result in the resonant case has been first proven by Goldberg \cite{gold}. For further
results for continuous Schr\"odinger equations see \cite{EHT}, \cite{EKMT}.

%%%%%%%%%%%%%%%%%%%%%%%%%%%%%%%%%%%%%%%%%%%%%%%%%%%%%%%%%%%%%%%%%%%%%%%%%%%%%
\section{Properties of the Scattering Matrix}
%%%%%%%%%%%%%%%%%%%%%%%%%%%%%%%%%%%%%%%%%%%%%%%%%%%%%%%%%%%%%%%%%%%%%%%%%%%%%

In this section we look at scattering theory for the Jacobi operator $H$. As a general reference we refer to \cite[Chapter~10]{tjac}. If \eqref{mainas} is contained in $\ell^1_{1}(\Z)$
then there exist Jost solutions $\vp_\pm(z,n)$ of 
\beq \label{josteq}
H \vp_\pm(z,n)= \frac{z+z^{-1}}{2} \vp_\pm(z,n), \qquad 0< |z| \le 1,
\eeq
which satisfy $\lim_{n \to \pm \infty} \widetilde{\vp}_\pm(z,n)=1$, where $\widetilde{\vp}_\pm(z,n)=\vp_\pm(z,n) z^{\mp n}$. These Jost solutions can equivalently be expressed as 
\beq \label{jostker}
\vp_\pm(z,n)=\sum^{\pm\infty}_{\ell=n} K_\pm(n,\ell)z^{\pm \ell}, \quad n \in \Z, \quad |z| \le 1 ,
\eeq
where the transformation operators satisfy
\beq \label{trans}
|K_\pm(n,\ell)|\leq
C_\pm(n) \left( \delta(n,\ell) + (1-\delta(n,\ell)) \sum^{\pm\infty}_{k=\floor{\frac{n+\ell}{2}}}\Big(\Big|a(k)-\frac 1 2 \Big|+
|b(k)|\Big) \right)
\eeq
for $\pm \ell\geq \pm n$ and $\floor{.}$ denotes the usual floor function.

If furthermore $\pm n \geq \mp 1$ holds, we can replace $C_\pm(n)$ by some universal constant $C_\pm(n) \le C$. For \[\widetilde{\vp}_\pm(z,n)=\sum^{\pm\infty}_{\ell=0 }K_\pm(n, \ell +n)z^{\pm \ell}=\sum^{\pm\infty}_{\ell=0 }\Kt_\pm(n, \ell)z^{\pm \ell}\] we get similar estimates. Introduce the Wiener algebra
\[
\A= \Big\{ f(z) = \sum_{m\in\Z} \hat{f}(m) z^m\  \Big|\, \|\hat{f}\|_{\ell^1} < \infty, \ \ |z|=1 \Big\}, 
\]
with the norm $\Vert f \Vert_\A=\Vert \hat f\Vert_{\ell^1}$.
Formulas \eqref{jostker} and \eqref{trans} imply
\beq \label{vpinA}
\vp_\pm(z,n), \widetilde{\vp}_\pm(z,n) \in \A,
\eeq
with $\Vert \vpt(z,n) \Vert_{ \A }$ independent of $n$ for $\pm n \geq 0$. Let us also recall the discrete version of the Wiener lemma \cite{Wiener} which states that if $f(z) \in \A $ and $f(z)\not =0$ for all $|z|=1$ then $f^{-1}(z)\in\A$. The estimate \eqref{trans} immediately leads to the following property: 

\begin{lemma} \label{lem1rcdisc}
If \eqref{mainas} holds for $\sigma=j+1$ with $j\geq 0$, then $\frac{d^l}{d z^l}(\vpt_\pm(z,n)) $ is an element of $\A$ for $0 \le l \le j$. Moreover, for $\pm n \geq 0$ the $\A$-norms of these expressions do not depend on $n$. 
\end{lemma}

The fact that $\vp_\pm(z^{-1},n)$ also solves the equation \eqref{josteq} gives rise to the scattering relations
\beq \label{scatdisc}
T(z) \vp_\pm(z,n)= R_\mp(z) \vp_\mp(z,n) + \vp_\mp(z^{-1},n), \qquad |z|=1,
\eeq
where the transmission coefficient $T$ and the reflection coefficients $R_\pm$ can be expressed in terms of Wronskians. 
To this end let us denote the discrete Wronskian by
\beq\label{alg2}
W(f(z,n), g(z,n))=a(n-1)\big(f(z,n-1)g(z,n)-g(z,n-1)f(z,n)\big). 
\eeq
Introducing the functions
\[
W(z)=W(\vp_+(z,1), \vp_-(z,1)), \: W_\pm(z)=W(\vp_\mp(z,1),\vp_\pm(z^{-1},1)),
\]
it follows that
\beq\label{defRRdisc}
T(z)= \frac{z-z^{-1}}{2 \I  W(z)},\quad R_\pm(z)= \mp\frac{W_\pm(z)}{W(z)}.
\eeq
In \cite[Theorem 4.1]{EKT} it was proved that the transmission and reflection coefficients
are elements of the Wiener algebra. Here we extend this result to derivatives.

\begin{theorem}\label{MT}
 If $|a(n)-\frac{1}{2}| + |b(n)|\in\ell^1_{j+1}(\Z)$, then $\frac{d^l}{dz^l}(T(z))\in \A$ and $\frac{d^l}{dz^l}R_\pm(z)\in\A$ for $0\le l \le j$. 
\end {theorem}

\begin{proof}
We  only focus on the resonant case $W(1)W(-1)=0$, since the other case is straightforward. Let $\hat{z}\in\{\pm 1\}$ be a point with $W(\hat{z})=0$.
As a first step we introduce the expressions
\[
\breve W_\pm(z)=\vp_\pm( z,1)\vp_\pm(\hat z,0)- \vp_\pm( z,0)\vp_\pm(\hat z,1).
\]
Then the calculations in \cite[Lemma ~4.1]{emt} show
\beq\label{hatW}
\breve W_\pm(z)=\zeta(z)\widetilde{\Psi}_\pm(z), \qquad  \zeta(z)=\frac{z-\hat z}{z},
\eeq
where
\begin{align} \label{defh}
\widetilde{\Psi}_\pm(z)=\sum_{\ell=\frac{1\pm 1}{2}}^{\pm\infty}h_\pm(\ell)(\hat z z)^{\pm \ell}, \quad 
h_\pm(\ell)=\Phi_\pm^{(1)}(\ell)\vp_\pm(\hat z,0) -\Phi_\pm^{(0)}(\ell)\vp_\pm(\hat z,1)
\end{align}
and
\[
\Phi_\pm^{(k)}(m)=\sum^{\pm\infty}_{\ell=m} K_\pm(k,\ell)\hat z^{\ell}.
\]
Moreover, $h_\pm(\ell)$ satisfies a similar estimate as \eqref{trans}, as we will show in Lemma \ref{lem3rcdisc} below.
This immediately leads to the conclusion that $\widetilde{\Psi}_\pm(z)$ and all its derivatives up to order $j$ are elements of the Wiener algebra. 
Next we need to distinguish between the cases $\vp_+(\hat z,0)\vp_-(\hat z,0)\neq 0$ and $\vp_+(\hat z,1)\vp_-(\hat z,1)\neq 0$
(these are the only cases since the solutions $\vp_\pm(\hat z,n)$ cannot vanish at two consecutive points). Straightforward computations show
\[
\frac{W(z)}{z-\hat{z}}=  \frac{a(0)}{z} \begin{cases} \left( \frac{\vpt_+(z,0)}{\vpt_-(\hat z,0)}
\widetilde{\Psi}_-(z)-\frac{\vpt_-(z,0)}{\vpt_+(\hat z,0)} 
\widetilde{\Psi}_+(z) \right), & \vp_+(\hat z,0)\vp_-(\hat z,0)\neq 0, \\
z \hat z \left( \frac{\vpt_+(z,1)}{\vpt_-(\hat z,1)}
\widetilde{\Psi}_-(z)-\frac{\vpt_-(z,1)}{\vpt_+(\hat z,1)} 
\widetilde{\Psi}_+(z)\right), & \vp_+(\hat z,1)\vp_-(\hat z,1)\neq 0.\end{cases}
\]
Since $|T(z)| \le 1$, the zeroes of $W(z)$ can at most be of first order, which shows that $\frac{W(z)}{z-\hat{z}}$ can vanish at most at $-\hat{z}$.
Hence if $W(-\hat{z})\ne 0$ then $\frac{z-\hat{z}}{W(z)}\in\mathcal{A}$ by Wiener's lemma (including all derivatives up to order $j$) and hence
the same is true for $T(z)$. Otherwise, if $W(-\hat{z})= 0$ and $W(\hat z) \neq 0$ one has the analogous properties for $\frac{W(z)}{z+\hat{z}}$ and using
a smooth cut-off function (which is one near $\hat{z}$ and vanishes near $-\hat{z}$) one can combine both results into one for
$\frac{W(z)}{(z-\hat{z})(z+\hat{z})}$ and proceed is before.

To get similar results for $R_\pm$, we use the following formulas: 
\[
\frac{W_\pm(z)}{z-\hat z}= \frac{a(0)}{z} \begin{cases} \left( \frac{\vpt_\mp(z,0)}{\vpt_\pm(\hat z,0)}
\widetilde{\Psi}_\pm(\frac{1}{z})-\frac{\vpt_\pm(z^{-1},0)}{\vpt_\mp(\hat z,0)} 
\widetilde{\Psi}_\mp(z) \right), & \vp_+(\hat z,0)\vp_-(\hat z,0)\neq 0,\\
\left( z \hat z\frac{\vpt_\mp(z,1)}{\vpt_\pm(\hat z,1)}
\widetilde{\Psi}_\pm(\frac{1}{z})-\frac{ \hat z}{z} \frac{\vpt_\pm(z^{-1},1)}{\vpt_\mp(\hat z,1)} 
\widetilde{\Psi}_\mp(z) \right), & \vp_+(\hat z,1)\vp_-(\hat z,1)\neq 0,\end{cases}
\]
and proceed as in the previous case.
\end{proof}

\begin{lemma} \label{lem3rcdisc}
For $h_\pm(m)$ given by \eqref{defh} we have 
\beq \label{defetadisc}
|h_\pm(m)| \leq \hat{C} \sum^{ \pm \infty}_{n=\floor{\frac{m\pm1}{2}}}\Big(\Big|a(n)-\frac 1 2 \Big|+
|b(n)|\Big)=\hat{C} \widetilde{\eta}_\pm(m)
\eeq
for some constant $\hat{C}>0$ and $\pm m \geq 0$.  
\end{lemma}

\begin{proof}
We only consider the minus case here. First of all we need the discrete version of the Marchenko equation, which is given by (\cite[Section ~10.3]{tjac}): 
\beq \label{marcheqdisc}
K_\pm(n,m) + \sum_{\ell=n}^{\pm \infty}K_\pm(n,\ell)F_\pm(\ell+m) =
\frac{\delta(n,m)}{K_\pm(n,n)}, \quad \pm m \geq \pm n,
\eeq
where for $F_\pm(\ell)$ we have the following estimate:
\beq \label{fest}
|F_\pm(\ell)|\leq C\sum^{ \pm \infty}_{n=\floor{\frac{\ell}{2}}}\Big(\Big|a(n)-\frac 1 2 \Big|+
|b(n)|\Big).
\eeq
After some calculations, which can be found in \cite[Lemma ~4.1]{emt}, the following equation for $h_-$ can be obtained in the case $\vp_-(\hat z,1)\vp_+(\hat z,1)\neq 0$ (the other case $\vp_-(\hat z,1)=\vp_+(\hat z,1)=0$ is similar, cf.\ \cite[Lemma ~4.1]{emt}): 
\beq\label{main3}
h_-(m) -\sum_{\ell=0}^{-\infty}
h_-(\ell)v(\ell+m+1)=-\frac{\delta(0,m)}{K_-(0,0)}\vp_-(\hat z,1), 
\eeq
where $v(\ell)=F_-(\ell)\hat z^{-\ell}$. Now let $-m \geq 2$. We rewrite \eqref{main3} as 
\begin{align*}
&h_-(m) -\sum_{\ell=N}^{-\infty} h_-(\ell)v(m+\ell+1)=-\widetilde{H}(m,N), \text{ where} \\
&\widetilde{H}(m,N)=-\frac{\delta(0,m)}{K_-(0,0)}\vp_-(\hat z,1)+\sum_{\ell=0}^{N} h_-(\ell)v(\ell+m+1)
\end{align*}
and $N \le 0$ such that $C \sum_{\ell=N}^{-\infty} \widetilde{\eta} (\ell)<1$ with $C$ given by $\eqref{fest}$. The estimate
\beq \label{Htest} 
|\widetilde{H}(s,N)| \le C(N) \widetilde{\eta}_-(m), \quad m \le -2,
\eeq
follows from monotonicity of $\widetilde{\eta}_-$ and $h_-(\cdot) \in \ell^\infty(\mathbb{Z}_-)$. Now we set
\[
h_{-, 0}(m)=\widetilde{H}(m,N), h_{-,k+1}(m)=\sum_{\ell=N}^{-\infty} h_{-, k}(\ell) v(\ell+m+1) . 
\]
We show that
\beq \label{hnest} 
\left| h_{-,k}(m) \right| \leq C(N) \widetilde{\eta}_-(m) \left(  C \sum_{\ell=N}^{-\infty} \widetilde{\eta}_- (\ell)  \right)^k, 
\eeq
where $C(N)$ is given by \eqref{Htest}. But this easily follows by induction, again using monotonicity of $\widetilde{\eta}_-$ and boundedness of $h_-$. 
\end{proof}

For later use we note that in the resonant case the Jost solutions are dependent at $\hat{z}$. If we define $\gamma$ via
\beq\label{def:gam}
\vp_+(\hat{z},n) = \gamma \vp_-(\hat{z},n),
\eeq
 then a straightforward calculation using the scattering relations \eqref{scatdisc} as well as $|T(k)|^2+|R_\pm(k)|^2=1$
shows
\beq\label{eq:trz}
T(\hat{z})= \frac{2\gamma}{1+\gamma^2}, \qquad R_\pm(\hat{z})= \pm\frac{1-\gamma^2}{1+\gamma^2}.
\eeq
In particular, all three quantities are real-valued since $\vp_\mp(\hat{z},n)\in\R$ and hence $\gamma\in\R$.
In the non-resonant case the scattering relations show
\beq
T(\hat{z})= 0, \qquad R_\pm(\hat{z})= -1.
\eeq

To establish Theorem \ref{mainthmdisc}, we also need the following generalization of Lemma~\ref{lem1rcdisc}:

\begin{lemma} \label{lem2rcdisc}
Let \eqref{mainas} be contained in $\ell_{j+1}^1(\Z)$. Then $\frac{d^l}{d z^l}(\frac{\vpt_\pm(z,n)-\vpt_\pm(\hat z,n)}{z-\hat z}) \in \A$ for $0\le l \le j-1$.
Moreover, for $\pm n \le \mp 1$, the $\A$-norms of these expressions do not depend on $n$. 
\end{lemma}

\begin{proof}
By \eqref{jostker} we have that 
\begin{align*}
\frac{\vpt_\pm(z,n)-\vpt_\pm(\hat z,n)}{z-\hat z}&= \sum_{\ell=\pm 1}^{\pm \infty} \Kt_\pm(n,\ell) \frac{z^{\pm \ell}-\hat z^{\pm \ell}}{z-\hat z}=\sum_{\ell=\pm 1}^{\pm \infty} \Kt_\pm(n,\ell) \sum_{k=0}^{\pm \ell -1} z^k \hat{z}^{\pm \ell-1-k} \\
&=\sum_{\ell=0}^{\pm \infty} (\sum_{k=\ell \pm 1}^{\pm \infty} \Kt_\pm(n,k) \hat{z}^{\ell+k-1}) z^{\pm \ell}.
\end{align*}
Using this  we conclude that the corresponding derivatives of the previous expressions have an $\A$-norm bound independent from $n$, if $\pm n \le \mp 1$ by \eqref{mainas}.
\end{proof}

In a similar way, we obtain the following lemma:

\begin{lemma} \label{lem4rcdisc}
Let \eqref{mainas} be contained in $\ell_{j+1}^1(\Z)$. Then $\frac{d^l}{d z^l}(\frac{\widetilde{\Psi}_\pm(z)-\widetilde{\Psi}_\pm(\hat z)}{z-\hat z})$ as well as $\frac{d^l}{d z^l}(\widetilde{\Psi}_\pm(z))$ are elements of $\A$ for $0\le l \le j-1$, where $\widetilde{\Psi}_\pm(z)$ are defined in \eqref{defh}. 
\end{lemma}

\begin{proof}
This follows as in the previous lemma using the estimate for $h_\pm$ from Lemma~\ref{lem3rcdisc}.
\end{proof}

Combining the last results we obtain:

\begin{theorem} \label{thm6rcdisc}
Let \eqref{mainas} be contained in $\ell_{j+1}^1(\Z)$ and let $\hat z\in\{-1, 1\}$.  Then \[\frac{d^l}{d z^l}\left(\frac{T(z)-T(\hat z)}{z-\hat z}\right)\in\A,\ \ \frac{d^l}{d z^l}\left(\frac{R_\pm(z)-R_\pm(\hat z)}{z-\hat z}\right) \in\A\ \ \mbox{ for}\ 0\le l \le j-1.\]
\end{theorem}

%%%%%%%%%%%%%%%%%%%%%%%%%%%%%%%%%%%%%%%%%%%%%%%%%%%%%%%%%%%%%%%%%%%%%%%%%%%%%
\section{Dispersive Decay in the Discrete Case}
%%%%%%%%%%%%%%%%%%%%%%%%%%%%%%%%%%%%%%%%%%%%%%%%%%%%%%%%%%%%%%%%%%%%%%%%%%%%%

In this section we prove Theorem \ref{main1} and Theorem \ref{mainthmdisc}. For the one-parameter group of \eqref{jaceq} the spectral theorem  and Stone's formula imply
\beq\label{PP}
   \E^{-\I tH}P_{ac}
   =\frac 1{2\pi \I}\int\limits_{-1}^{1}
   \E^{-\I t\omega}(\cR(\omega+\I 0)- \cR(\omega-\I 0))\,d\omega,
\eeq
where $\cR(\omega)=(H-\omega)^{-1}$ is the resolvent of the Jacobi operator $H$
and the limit is understood in the strong sense \cite[Problem 4.3]{tschroe}.
For the kernel of the resolvent $R(z)=(H-\frac{z+z^{-1}}{2})^{-1}$, we have the following explicit formula for $0 < |z| \le 1$ (cf.\ \cite[(1.99)]{tjac}): 
\beq \label{reskerdisc}
[R(z)](n,k) = \frac{1}{W(z)} \left\{ \begin{array}{cc}
\vp_+(z,n) \vp_-(z,k)
\;\; \mbox{for} \;\; n \ge k, \\[2mm]
\vp_+(z,k) \vp_-(z,n)
 \;\; \mbox{for} \;\; n\le k. \end{array} \right.
\eeq
Formulas \eqref{PP} and \eqref{reskerdisc} then lead to the following explicit representation of the kernel of the propagator $\E^{- \I t H}P_{ac}$:
\beq \label{mainkerdisc3} 
\left[ \E^{-\I t H}P_{ac} \right](n,k) = 
\frac{1}{2 \pi} \int_{-\pi}^{\pi} \E^{-\I t \cos \theta}
\vp_+(\E^{\I\theta},n) \vp_-(\E^{\I\theta},k) T(\E^{\I\theta}) d\theta.
\eeq
We also need a small variant of the van der Corput lemma, which can be found in \cite[Lemma ~5.1]{EKT}:

\begin{lemma} \label{vdcorput}
Consider the oscillatory integral
$I(t) = \int_a^b \E^{\I t v(\theta)} f(\theta) d\theta$ with $[a,b]\subset\R$ some compact interval
and $v(\theta)$ real-valued.
Let $\min\limits_{\theta\in[a,b]}|v^{(j)}(\theta)|=m_j>0$ for some $j\ge 2$ and let 
\beq\label{defAab}
f(\theta)=\sum_{m\in\Z} \hat f (m) \E^{\I m\theta } \quad \mbox{for } \theta\in[a,b] \quad \mbox{with }  \|f\|_1= \sum_{m\in\Z}|\hat f(m)|<\infty.
\eeq
Then
\beq\label{mainest2}
|I(t)| \le \frac{C_j \|\hat{f}\|_1}{(m_j t)^{1/j}}, \ \mbox{for}\  t\ge 1,
\eeq
where $C_j$ is an universal constant.
\end{lemma}

The proofs of Theorem \ref{main1} and the non-resonant case of Theorem~\ref{mainthmdisc} can now be done in exactly the same
way as in \cite[Theorem 5.3, Theorem 6.1]{EKT}.
Thus it remains to consider the resonant case of our main Theorem~\ref{mainthmdisc}. 
As a first step we have a closer look at our projection operators in the next lemma:

\begin{lemma} \label{lem5rcdisc}
The following expressions for the kernels of our Projectors $P_{\hat z}$ are valid:
\begin{align} \label{lem5rcdisc1}
&\frac{\E^{-\I t}}{\sqrt{-2\pi \I t}}[P_1](n,k)=\vp_+(1,n) \vp_-(1,k) T(1) \left[ \frac{1}{2 \pi}\int_{-\frac{\pi}{2}}^{\frac{\pi}{2}} \E^{-\I t \cos \theta} d \theta - \frac{1}{\I t \pi}\right]+\mathcal{O}(t^{-\frac{3}{2}}),  \\
&\frac{\E^{\I t}}{\sqrt{2\pi \I t}}[P_{-1}](n,k)=\vp_+(-1,n) \vp_-(-1,k) T(-1) \left[ \frac{1}{2 \pi}\int_{\frac{\pi}{2}}^{\frac{3 \pi}{2}} \E^{-\I t \cos \theta} d \theta + \frac{1}{\I t \pi} \right]+\mathcal{O}(t^{-\frac{3}{2}}). \nn
\end{align}
\end{lemma}

\begin{proof}
It is clear that $\vp_{\hat{z}}(n) = c_\pm \vp_\pm(\hat{z},n)$ and by our normalization $c_-^{-2}+c_+^{-2} = 2$. Using \eqref{def:gam} we have $c_-= \gamma c_+$ and
hence $c_\pm^2=\frac{1+\gamma^{\mp1}}{2}$. Moreover, \eqref{eq:trz} implies $c_+ c_- T(\hat{z})=1$ and hence $[P_{\hat z}](n,k)=\vp_+(\hat z,n) \vp_-(\hat z,k) T(\hat z)$.
Furthermore, by \cite[10.9]{dlmf} and \cite[11.5]{dlmf}, we get
\[
\frac{1}{2}(J_0(t)- \I H_0(t))=\frac{1}{2 \pi}\int_{-\frac{\pi}{2}}^{\frac{\pi}{2}} \E^{-\I t \cos \theta} d \theta, 
\]
where $J_0(t)$ denotes the Bessel function, and $H_0(t)$ the Struve function, both of order $0$ (cf. \cite[10.2]{dlmf} and \cite[11.2]{dlmf}). Moreover, by \cite[11.6]{dlmf} the following asymptotics hold
\[
H_0(t)-Y_0(t)=\frac{2}{\pi t}+\mathcal{O}(t^{-3}),
\]
where $Y_0(t)$ denotes the Neumann function of order $0$ (cf.\ \cite[10.2]{dlmf}). Since by \cite[10.4]{dlmf} $H_{0}^{(1)}(t)=J_0(t)+ \I Y_0(t)$ and $H_{0}^{(2)}(t)=J_0(t)- \I Y_0(t)$, the claim follows using the asymptotics \cite[10.17.5, 10.17.6]{dlmf}. The argument for the second formula is analogous.
\end{proof}

Since we have got everything we need,we can finish the proof of our main theorem now: 

\begin{proof}[Proof of Theorem \ref{mainthmdisc}]
Using \eqref{mainkerdisc3}, we get 
\[
[\E^{-\I t H}P_{ac}](n,k)=\frac{1}{2 \pi} \int_{-\pi}^{\pi} \E^{-\I t \cos \theta} G(\theta,n,k) d\theta,
\]
where
\beq \label{bigG}
G(\theta, n, k)=\E^{\I |n-k| \theta} \vpt_+(\E^{\I\theta},k) \vpt_-(\E^{\I\theta},n)T(\E^{\I\theta}).
\eeq
Taking advantage of the fact that all functions here are $2 \pi$-periodic, we can integrate along the interval $[-\frac{\pi}{2}, \frac{3 \pi}{2}]$ instead. 
By Lemma~\ref{lem5rcdisc},
\begin{align}\nn
&\mathcal P(n,k,t):=[\E^{-\I t H}P_{ac}-\frac{\E^{-\I t}}{\sqrt{-2\pi \I t}}P_1-\frac{\E^{\I t}}{\sqrt{2\pi \I t}}P_{-1}](n,k)\\ \label{discmain1}
& \quad=\frac{1}{2 \pi}\int_{-\frac{\pi}{2}}^{\frac{\pi}{2}} \E^{-\I t \cos \theta} \big(G(\theta, n, k)-G(0, n, k) \big) d \theta + G(0, n, k)\frac{1}{\I t \pi} \\ \label{discmain2}
& \qquad +\frac{1}{2 \pi}\int_{\frac{\pi}{2}}^{\frac{3 \pi}{2}} \E^{-\I t \cos \theta} \big(G(\theta, n, k)-G(\pi, n, k) \big) d \theta - G(\pi, n, k)\frac{1}{\I t \pi} + \mathcal{O}(t^{-\frac{3}{2}}). 
\end{align} 
If we integrate \eqref{discmain1} by parts, we obtain
\begin{align*}
&\frac{1}{2 \pi}\int_{-\frac{\pi}{2}}^{\frac{\pi}{2}} \E^{-\I t \cos \theta} \big(G(\theta, n, k)-G(0, n, k) \big) d \theta + G(0, n, k)\frac{1}{\I t \pi}  = \frac{1}{2 \pi \I t} G\left(\frac{\pi}{2}, n, k\right)\\
&+\frac{1}{2 \pi \I t} G\left(-\frac{\pi}{2}, n, k\right)
 - \frac{1}{2 \pi \I t} \int_{-\frac{\pi}{2}}^{\frac{\pi}{2}} \E^{-\I t \cos \theta} \frac{d}{d \theta}\left( \frac{G(\theta, n, k)-G(0, n, k)}{\sin \theta} \right) d \theta.
\end{align*}
If we do the same in \eqref{discmain2} and use $G(\frac{3\pi}{2}, n, k)=G(-\frac{\pi}{2}, n, k)$ we see that all the terms of order $\frac 1 t$ vanish and hence 
\begin{align} \label{maintermdisc}
 \mathcal P(n,k,t)&=\frac{\I}{2 \pi t} \int_{-\frac{\pi}{2}}^{\frac{\pi}{2}} \E^{-\I t \cos \theta} \frac{d}{d \theta}\left( \frac{G(\theta, n, k)-G(0, n, k)}{\sin \theta} \right) d \theta  \\ \nn
&+\frac{\I}{2 \pi t} \int_{\frac{\pi}{2}}^{\frac{3\pi}{2}} \E^{-\I t \cos \theta} \frac{d}{d \theta}\left( \frac{G(\theta, n, k)-G(\pi, n, k)}{\sin \theta} \right) d \theta +  \mathcal{O}(t^{-\frac{3}{2}}),
\end{align}
where we neglect the summand of order $t^{-\frac{3}{2}}$ from now on. 
The other two summands are treated separately and we show the desired $t^{-\frac{4}{3}}$ time decay for each expression. 
Since the calculations are similar, we only focus on \eqref{maintermdisc}. 
Next, to apply Lemma \ref{vdcorput}, we split the domain of integration into parts where either the second or third derivative of the phase $- \I t \cos \theta$ is nonzero. 
This gives us the time decay. It remains to show that the derivatives with respect to $\theta$ are elements of $\A$ with $\A$-norms at most proportional to $(|n|+|k|)^2$.
Of course this will in general not be true because of the possible singularity at the other end of the continuous spectrum.
However, since this possible singularity is outside the domain of integration we can redefine our functions there. 

We distinguish the cases (i) $n \le 0 \le k$, (ii) $0 \le n \le k$, and (iii) $n \le k \le 0$.
Introduce the functions
\[
f(\theta)=\frac{\E^{\I\theta}-1}{\sin\theta}, \:\: g(\theta,n,k)=G(\theta,n,k)\E^{-\I |n-k|\theta}, \:\: h(\theta,n,k)= \frac{g(\theta,n,k)-g(0,n,k)}{\E^{\I\theta}-1}.
\]
Then
\begin{align*}
& \frac{d}{d\theta} \frac{G(\theta, n, k)-G(0, n, k)}{\sin \theta}= \frac{d}{d\theta}\left( f(\theta) \frac{\E^{\I (k-n) \theta}-1}{\E^{\I\theta}-1} g(\theta,n,k) +
f(\theta)  h(\theta,n,k)\right)\\
\end{align*}
Since only the values of $f$ for $\theta\in [-\frac{\pi}{2},\frac{\pi}{2}]$ are relevant, we redefine it on $[\frac{\pi}{2},\frac{3\pi}{2}]$ such that
$f,f'\in \A$ (e.g.\ by multiplying with a smooth periodic function which is $1$ on $[-\frac{\pi}{2},\frac{\pi}{2}]$ and vanishes near $\pi$).

We observe that
\[
\frac{\E^{\I (k-n) \theta}-1}{\E^{\I\theta}-1}=\frac{z^{k-n}-1}{z-1}= \sum_{\ell=0}^{k-n-1} z^{\ell} \in \A
\]
with its $\A$-norm bounded by $k-n$, and that of its derivative by $(k-n)^2$. Now it remains to consider the $\A$-norms of $g(\theta,n,k)$, $\frac{d}{d \theta} g(\theta,n,k)$, $h(\theta,n,k)$, and $\frac{d}{d \theta} h(\theta,n,k)$.

We start the with case (i). Then $g(\theta,n,k) \in\A$ with $\A$-norm independent of $n$ and $k$.
After applying the product rule, Lemma~\ref{lem1rcdisc} and Theorem~\ref{MT} imply $\|\frac{d}{d \theta} g(\theta,n,k)\|_{\A}\le C$.
Moreover, by \eqref{bigG}
\begin{align*}
h(\theta,n,k)&=\frac{T(z)-T(1)}{z-1}\vpt_+(z,k) \vpt_-(z,n) 
+\frac{\vpt_+(z,k)-\vpt_+(1,k)}{z-1}\vpt_-(z,n)T(1)\\ &+ \frac{\vpt_-(z,n)-\vpt_-(1,n)}{z-1}\vpt_+(z,k)T(1),\quad z=\E^{\I\theta}.
\end{align*}

Invoking Lemma \ref{lem1rcdisc}, Lemma~\ref{lem2rcdisc}, and Theorem~\ref{thm6rcdisc}, proves that $h$ and its derivative are again contained in $\A$, with $\A$-norms independent of $n$ and $k$, thus we are done in case (i).
In the cases (ii) and (iii) we use the scattering relations \eqref{scatdisc} to get the following representations: 
\[
g(\theta,n,k)=\begin{cases}
 \vpt_+(z,k)\left(R_+(z)\vpt_+(z,n)z^{2n} + \vpt_+(z^{-1},n)\right), & 0\leq n\leq k,  \\[2mm]
 \vpt_-(z,n)\left(R_-(z)\vpt_-(z,k)z^{2k} + \vpt_+(z^{-1},k)\right), & n\leq k\leq 0, 
\end{cases}
\]
where $z=\E^{\I \theta}$. 
Thus $g(\theta,n,k)$ has an $\A$-norm independent of $n$ and $k$, since for any $f\in \A$ and $m \in \Z$, we have that $f(z)z^m \in \A$,
with $\A$-norm independent of $m$. Taking derivatives with respect to $\theta$ we get additional terms from the derivatives of $z^{2n}$ and $z^{2k}$,
so by Lemma~\ref{lem1rcdisc} and Theorem~\ref{MT} it follows that  $\| \frac{d}{d \theta} g(\theta,n,k) \|_{\A} $ is at most proportional to $|n|$ and $|k|$ respectively. 
Finally it remains to consider $h(\theta,n,k)$. In case (ii) $h(\theta,n,k)$ can be represented as follows:
\begin{align*}
h(\theta,n,k)=\frac{\vpt_+(z,k)-\vpt_+(1,k)}{z-1}\vpt_+(z^{-1},n) &+ \frac{\vpt_+(z^{-1},n)-\vpt_+(1,n)}{z-1}\vpt_+(1,k)\\
+\frac{R_+(z)-R_+(1)}{z-1} \vpt_+(z,n) \vpt_+(z,k) z^{2 k}&+\frac{z^{2k}-1}{z-1} R_+(1)\vpt_+(z,n)\vpt_+(z,k)\\
+\frac{\vpt_+(z,n)-\vpt_+(1,n)}{z-1}\vpt_+(z,k)R_+(1) &+ \frac{\vpt_+(z,k)-\vpt_+(1,k)}{z-1}\vpt_+(1,n)R_+(1) , 
\end{align*}
for $z=\E^{\I \theta}$. 
Here again every summand is an element of $\mathcal A$ by Lemma~\ref{lem1rcdisc}, Lemma~\ref{lem2rcdisc}, and Lemma~\ref{thm6rcdisc}. 
Since the derivative of $\frac{z^{2k}-1}{z-1}$ also occurs here, we conclude that the $\mathcal A$-norm of $\frac{\partial}{\partial \theta} h(\theta,n,k)$ is at most proportional to $|k|^2$. 
In the case (iii) this derivative will be proportional to $|n|^2$. 
This indeed proves the claim about $ \frac{G(\theta, n, k)-G(0, n, k)}{\sin \theta} $ for $\theta \in [-\frac{\pi}{ 2}, \frac{\pi}{2}]$. 
Since the same procedure works for $\frac{G(\theta, n, k)-G(\pi , n, k)}{\sin \theta}$, with $\theta \in [\frac \pi 2, \frac{3\pi}{2}]$ this finishes the proof.
\end{proof}

%%%%%%%%%%%%%%%%%%%%%%%%%%%%%%%%%%%%%%%%%%%%%%%%%%%%%%%%%%%%%%%%%%%%%
%\bigskip
\noindent{\bf Acknowledgments.}
I.E.~is indebted to the Department of Mathematics at the University of Vienna for its hospitality and support during the
fall of 2014, where some of this work was done.
%%%%%%%%%%%%%%%%%%%%%%%%%%%%%%%%%%%%%%%%%%%%%%%%%%%%%%%%%%%%%%%%%%%%%

\end{document}